\documentclass{amsart}
\usepackage{graphicx} 
\usepackage{enumerate}
\usepackage{hyperref}
\usepackage{amssymb,amsmath,amsfonts,amsthm}
\usepackage{ytableau}
\usepackage{young}

\newcommand{\Irr}{\operatorname{Irr}}
\newcommand{\Sp}{\operatorname{Sp}}
\newcommand{\LR}{\operatorname{LR}}
\newcommand{\sgn}{\operatorname{sgn}}

\newtheorem{lemma}{Lemma}[section]
\newtheorem{theorem}{Theorem}

\newtheorem{prb}{Problem}

\title[On eigenvalues of permutations in representations of $S_n$ and $A_n$]{On eigenvalues of permutations in irreducible representations of symmetric and alternating groups}
\author{Alexey Staroletov}\thanks{The work was supported by the Russian Science Foundation, project 24-11-00119.}

\date{}

\begin{document}

\newcommand{\Address}{
{
\bigskip\noindent
\footnotesize
Alexey~Staroletov, \textsc{Sobolev Institute of Mathematics, Novosibirsk, Russia;}\\\nopagebreak
\textit{E-mail address: } \texttt{staroletov@math.nsc.ru}
\\\nopagebreak
}}
\begin{abstract}
Denote the symmetric group of degree $n$ by $S_n$. 
Let $\rho$ be an irreducible representation of $S_n$ over the field of complex numbers and $\sigma\in S_n$. 
In this paper, we describe the set of eigenvalues of $\rho(\sigma)$. Based on this result, we also obtain a description in the case of alternating groups.
\end{abstract}

\maketitle


{\bf MSC classes:} 20C30; 20C15; 05E10 

\section{Introduction}

Let $G$ be a finite group and $\mathbb{F}$ an algebraically closed  field. 
Suppose that $\rho$ is an irreducible representation of $G$ over $\mathbb{F}$. 
If $g\in G$, then the natural question is to describe the set of eigenvalues of $\rho(g)$. Denote by $\deg(\rho(g))$ the degree of the minimal polynomial of the matrix $\rho(g)$ and by $o(g)$ the order of $g$ modulo $Z(G)$.
The general problem was formulated in \cite{TZ08} as follows.
\begin{prb}\label{p:1} Determine all possible values for $\deg(\rho(g))$, and if possible, all triples
$(G, \rho, g)$ with $deg(\rho(g))<o(g)$, in the first instance under the condition that $o(g)$
is a $p$-power.
\end{prb}

There have been many papers devoted to this problem. The greatest progress has been achieved in the case of quasi-simple and almost simple groups, where elements of prime power orders are considered (see \cite{TZ22}, \cite{Z08} and references therein).

In this article, we are primarily interested in the symmetric and alternating groups of degree $n$, which are denoted by $S_n$ and $A_n$, respectively. The minimal polynomials of prime order elements of $A_n$ and $S_n$ in the ordinary or projective representations were found in \cite{Zal96}.
For the algebraically closed fields of positive characteristic $p$, Kleshchev and Zalesski \cite{KleshZal04} described the minimal polynomials of order $p$ elements in the irreducible representations of covering groups of $A_n$.

There are several results that consider not only the elements of prime power order.   The minimal polynomials of powers of cycles were found independently in \cite{Sw18} and \cite{YS21}.
Recently, the minimal polynomials of permutations of $A_n$ and $S_n$ such that in the cycle type the length of each cycle divides the length of the largest cycle have been described~\cite{V25}. In particular, this set of permutations contains all elements of prime power order.

In~\cite{ST89}, Stembridge gave combinatorial description of eigenvalues for permutations in every 
ordinary irreducible representation of $S_n$
in terms of so-called $\mu$-indices of standard tableaux. 
Later, a new approach to this result was presented by 
J{\"o}llenbeck and Schocker in~\cite{JS2000}. In a recent paper~\cite{PV-alt}, it is shown how Stembridge's description of eigenvalues can be extended to obtain a description of eigenvalues in the case of alternating groups.
In this paper, we describe eigenvalues in terms of cycle lengths.
\begin{theorem}\label{t:main}
Let $n\geqslant2$ be an integer.
Suppose that $\sigma\in S_n$ is a permutation with cycle type $\mu$ and $\rho:S_n\rightarrow GL(V)$ is an irreducible representation of $S_n$ over $\mathbb{C}$
which corresponds to a partition $\lambda$, where $\lambda\neq(n), (1^n)$. Denote the minimal polynomial of $\rho(\sigma)$ by $p_\lambda^\mu(x)$ and the set of roots of $p_\lambda^\mu(x)$ by 
$\Sp_\lambda(\mu)$. Then $p_\lambda^\mu(x)\neq x^{o(\sigma)}-1$ if and only if one of the following statements holds.
\begin{enumerate}[$(i)$]
\item There are no $t=n-\lambda_1$ elements in $\mu$
whose least common multiple equals $o(\sigma)$.
In this case, we have $\Sp_\lambda(\mu)=\{
\eta_{1}\cdots\eta_{t}~|~\eta_j^{\mu_{i_j}}=1, \text{ where } 
1\leqslant j\leqslant t\text{ and }\mu_{i_j}\in\mu\}$.
\item 
There are no $t=n-\lambda_1'$ elements in $\mu$
whose least common multiple equals $o(\sigma)$.
In this case, we have
$\Sp_\lambda(\mu)=\{
\sgn(\sigma)\eta_{1}\cdots\eta_{t}~|~\eta_j^{\mu_{i_j}}=1, \text{ where } 
1\leqslant j\leqslant t\text{ and }\mu_{i_j}\in\mu\}$.
\item $n\geqslant3$, $\lambda=(n-1,1)$,  $\mu=(n)$, $p_\lambda^\mu(x)=\frac{x^n-1}{x-1}$.
\item $n\geqslant4$, $\lambda=(2,1^{n-2})$, $\mu=(n)$, $p_\lambda^\mu(x)=\frac{x^n-1}{x+(-1)^n}$.
\item
$n\geqslant5$ is odd, $\lambda=(2,2,1^{n-4})$, $\mu=(n-2,2)$, $p_\lambda^\mu(x)=\frac{x^{2n-4}-1}{x-1}$.
\item
$n\geqslant5$ is odd, $\lambda=(n-2,2)$, $\mu=(n-2,2)$, $p_\lambda^\mu(x)=\frac{x^{2n-4}-1}{x+1}$.
\item $\lambda=(2,2)$, $(\mu,p_\lambda^\mu)\in\{((4), x^2-1), ((3,1), x^2+x+1), ((2,2), x-1)\}$;
\item $\lambda=(3^2)$, $\mu=(6)$, $p_\lambda^\mu(x)=\frac{x^6-1}{x^2+x+1}$.
\item $\lambda=(2^3)$, $\mu=(6)$, $p_\lambda^\mu(x)=\frac{x^6-1}{x^2-x+1}$.
\item $(\lambda,\mu)\in\{((2^3), (3,2,1)), ((4^2),(5,3)), ((2^4),(5,3)), ((2^5), (5,3,2))\}$
and $p_\lambda^\mu(x)=\frac{x^{o(\sigma)}-1}{x-1}$.
\item $(\lambda,\mu)\in\{((3^2), (3,2,1)), ((5^2), (5,3,2))\}$
and $p_\lambda^\mu(x)=\frac{x^{o(\sigma)}-1}{x+1}$.
\end{enumerate}
\end{theorem}

This gives a complete answer to Problem~\ref{p:1} in the case of symmetric groups and $\mathbb{F}=\mathbb{C}$. As a consequence, we see that the description is especially simple in some cases.
For example, in the above-mentioned result, where all cycle lengths divide the length of the largest one, the order of the permutation is equal to the length of the largest cycle and hence items $(i)$ and $(ii)$ are possible only for the trivial and sign representations. Similarly, if $\eta$ is a complex number such that $\eta^{p^k}=1$ for a prime power $p^k$ dividing the order of $\sigma$, then $\eta$ lies in the spectrum of $\rho(\sigma)$, except perhaps for the short list in items $(iii)-(xi)$, in particular, $p^k\leqslant4$. This observation generalizes the results for $p^k=2$
obtained in~\cite{PPV24} for $\eta=1$ and in~\cite{V25} for $\eta=-1$.

Using Theorem~\ref{t:main}, we also obtain a description of the eigenvalues for permutations in irreducible representations of $A_n$. Recall that if $\lambda\neq\lambda'$,
then $\chi_\lambda:=\chi^\lambda\downarrow^{S_n}_{A_n}$ is an irreducible character of $A_n$,
while if $\lambda=\lambda'$, then the restriction
$\chi^\lambda\downarrow^{S_n}_{A_n}$ is the sum of two irreducible characters of $A_n$, which we denote by $\chi_\lambda^+$ and $\chi_\lambda^-$. 
It is also well known that if the cycle type $\mu$ consists of odd distinct numbers, then the corresponding permutations form two conjugate classes in $A_n$, which are denoted by $\mu^+$ and $\mu^-$.
\begin{theorem}\label{t:main2}
Let $n\geqslant3$ be an integer. Suppose that $\sigma\in A_n$ is a permutation with cycle type $\mu$ and $\rho:A_n\rightarrow GL(V)$ is a nontrivial irreducible representation of $A_n$ over $\mathbb{C}$ which corresponds to a partition $\lambda$ with $\lambda_1\geqslant\lambda_1'$. Denote the character of $\rho$ by $\chi$. Denote the minimal polynomial of $\rho(\sigma)$ by $p(x)$ and the set of roots of $p(x)$ by 
$\Sp(\mu)$. Then $p(x)\neq x^{o(\sigma)}-1$ if and only if one of the following statements holds.
\begin{enumerate}[$(i)$]
\item There are no $t=n-\lambda_1$ elements in $\mu$
whose least common multiple equals $o(\sigma)$.
In this case, we have $\Sp(\mu)=\{
\eta_{1}\cdots\eta_{t}~|~\eta_j^{\mu_{i_j}}=1, \text{ where } 
1\leqslant j\leqslant t\text{ and }\mu_{i_j}\in\mu\}$.
\item $n\geqslant5$ is odd, $\lambda=(n-1,1)$, $\mu\in\{(n)^+,(n)^-\}$, $p(x)=\frac{x^n-1}{x-1}$.
\item $\lambda=(4^2)$, $\mu\in\{(5,3)^+,(5,3)^-\}$, and $p(x)=\frac{x^{15}-1}{x-1}$.
\item $n=5$, $(\chi,\mu)\in\{(\chi_{(3,1,1)}^+,(5)^+),(\chi_{(3,1,1)}^-,(5)^-)\}$, $p(x)=\frac{x^5-1}{(x-\eta^2)(x-\eta^3)}$, where $\eta=\cos{\frac{2\pi}{5}}+i\sin\frac{2\pi}{5}$.
\item $n=5$, $(\chi,\mu)\in\{(\chi_{(3,1,1)}^+,(5)^-),(\chi_{(3,1,1)}^-,(5)^+)\}$, $p(x)=\frac{x^5-1}{(x-\eta)(x-\eta^4)}$, where $\eta=\cos{\frac{2\pi}{5}}+i\sin\frac{2\pi}{5}$.
\item $n=4$, $(\chi,\mu)\in\{(\chi_{(2,2)}^+,(3,1)^+),(\chi_{(2,2)}^-,(3,1)^-)\}$, $p(x)=x-\omega$, where $\omega=\frac{-1+i\sqrt{3}}{2}$.
\item $n=4$, $(\chi,\mu)\in\{(\chi_{(2,2)}^+,(3,1)^-),(\chi_{(2,2)}^-,(3,1)^+)\}$, $p(x)=x-\omega^2$, where $\omega=\frac{-1+i\sqrt{3}}{2}$.
\item $n=4$, $(\chi,\mu)\in\{(\chi_{(2,2)}^+,(2,2)),(\chi_{(2,2)}^-,(2,2))\}$, $p(x)=x-1$.
\item $n=3$, $(\chi,\mu)\in\{(\chi_{(2,1)}^+,(3)^+),(\chi_{(2,1)}^-,(3)^-)\}$, $p(x)=x-\omega$, where $\omega=\frac{-1+i\sqrt{3}}{2}$.
\item $n=3$, $(\chi,\mu)\in\{(\chi_{(2,1)}^+,(3)^-),(\chi_{(2,1)}^-,(3)^+)\}$, $p(x)=x-\omega^2$, where $\omega=\frac{-1+i\sqrt{3}}{2}$.

\end{enumerate}
\end{theorem}

\section{Notation and preliminaries}
\subsection{Partitions and irreducible characters of $S_n$}
In the basic concepts and notations of the representation theory of $S_n$ and $A_n$, we follow~\cite{JK}. Denote by $\mathbb{N}$ the set $\{1,2,\ldots\}$ of the natural numbers. 
If $n\in\mathbb{N}$, then a sequence $\lambda=(\lambda_1,\ldots,\lambda_k)$ of natural numbers is called a {\it partition} of $n$ if $\lambda_1\geqslant\lambda_2\geqslant\ldots\geqslant\lambda_k$ and $\sum\limits_{i=1}^k\lambda_i=n$. 
In this case, $n$ is called the {\it size} of $\lambda$ which is also denoted by $|\lambda|$. The set of all partitions of $n$ is denoted by $\mathcal{P}(n)$. If $\lambda\in\mathcal{P}(n)$, then we also write $\lambda\vdash n$. 
It is natural to represent partitions of $n$ 
with Young diagrams which consist of left justified array of boxes having $\lambda_i$ boxes in the $i$-th row from the top.
If $\lambda\in\mathcal{P}(n)$, then the Young diagram corresponding to $\lambda$ is denoted by $[\lambda]$.
Sometimes it is also convenient to consider the Young diagram
$[\lambda]$ as a subset of the Cartesian plane: 
$[\lambda]=\{(i,j)\in\mathbb{N}\times\mathbb{N}~|~1\leqslant i\leqslant k, 1\leqslant j\leqslant\lambda_i\}$.

The partition which corresponds to the transpose diagram of $[\lambda]$ is denoted by $\lambda'$ and called {\it conjugate} to $\lambda$. 

If $\nu\vdash m$ and $\lambda\vdash n$ , where $m\leqslant n$,
then we write $\nu\subseteq\lambda$ if $\nu_i\leqslant\lambda_i$ for every $i$. 
We say that a box of a diagram $[\lambda]$ is a {\it corner}
if it is located at the right end of a row and having no box below.
Note that after removing a corner we obtain a correct Young diagram of size $n-1$. Moreover, $\nu\subseteq\lambda$ if and only if $[\nu]$ can be obtained from $[\lambda]$ by removing corners step by step.
By definition, all corners of the diagram lie in rows of different lengths, so we can order the corners according to the length of the row: the corner lying in a longest row is called the first, the corner lying in a row whose length is the second longest is called the second, and so on. Note that $[\lambda]$ always has the first corner, and the second exists only if $\lambda$ contains at least two distinct elements.

The set of ordinary irreducible characters of a finite group $G$
is denoted by $\Irr(G)$. The set $\Irr(S_n)$ is naturally in bijection with $\mathcal{P}(n)$, the character corresponding to $\lambda\vdash n$ is denoted by $\chi^\lambda$.
In particular, if $\lambda=(n)$, then $\chi^\lambda$ is the principal character corresponding to the trivial representation of $S_n$. On the other hand,
if $\lambda=(1^n)=(n)'$, then $\chi^\lambda$ corresponds to the sign representation,
where $S_n\ni\sigma\mapsto \sgn(\sigma)$.
In general, it is true that $\chi^{\lambda'}=\chi^\lambda\cdot\chi^{(1^n)}$.

Let us briefly discuss irreducible representations of $A_n$. If $\lambda\neq\lambda'$, then $\chi_\lambda:=\chi^\lambda\downarrow A_n=\chi^{\lambda'}\downarrow A_n$ is irreducible.  
If $\lambda=\lambda'$, then $\chi^\lambda\downarrow A_n$ 
splits into two irreducible and conjugate characters denoted by $\chi_\lambda^{\pm}$. These two sets give a complete system of  irreducible characters of $A_n$. Recall that
if $\sigma\in A_n$ and its cycle type differs from $(a_1,\ldots,a_k)$,
where $a_1,\ldots, a_k$ are the lengths of hooks along the main diagonal of $[\lambda]$, then $\chi_\lambda^\pm(\sigma)=\frac{1}{2}\chi^\lambda(\sigma)$.
If the cycle type of $\sigma$ equals $(a_1,\ldots,a_k)$,
then we say that $\sigma$ has cycle type $(a_1,\ldots,a_k)^+$
if $\sigma$ is conjugate in $A_n$ with the permutation $(1,2,\ldots, a_1)(a_1+1,\ldots,a_1+a_2)\ldots(a_1+\ldots+a_{k-1}+1,\ldots,a_1+\ldots+a_k)$, and otherwise we say that $\sigma$ has cycle type $(a_1,\ldots,a_k)^-$. 
If $\sigma$ has cycle type $(a_1,\ldots,a_k)^+$, then 
by convention it is assumed that $\chi^{\pm}_\lambda(\sigma)=\frac{1}{2}(\epsilon\pm\sqrt{\epsilon a_1\cdots a_k})$, where $\epsilon=(-1)^{(n-k)/2}$. 
We refer the reader to \cite[Chapter~2]{JK} for a more detailed discussion.

\subsection{The Littlewood--Richardson rule}
Suppose that $m,n\in\mathbb{N}$ with $m\leqslant n$.
For two partitions $\gamma\vdash m$ and $\lambda\vdash n$
satisfying $\gamma\subseteq\lambda$, 
we define the skew diagram $[\lambda/\gamma]$ by deleting all the boxes in $[\lambda]$ which also belong to $[\gamma]$. 
By {\it a skew semi-standard tableau $T$ of shape $\lambda/\gamma$},
we mean a labeling of the boxes of $[\lambda/\gamma]$ with positive integers such that labels are non-decreasing along each row and strictly increasing down each column. In this case, the sequence $cont(T)=(c_1,\ldots,c_i,\ldots)$, where
$c_i$ is the number of times that $i$ is used in $T$, is called {\it the content} of $T$. For example,

\begin{center}
\ytableausetup{boxsize=1em}
$S=
\begin{ytableau}
   \none & \none & \none & 2 \\
  \none & \none  & 4 & 4  \\
  2 & 3 
\end{ytableau}$
\end{center}
is a skew semi-standard tableau of shape $\lambda/\gamma$, where 
$\lambda=(5, 4, 4, 2)$, $\gamma=(5, 3, 2)$, and $cont(S)=(0,2,1,2)$.

We can read entries of a skew semi-standard tableau $T$ from right to left in each row and one row after the other, downwards, obtaining a sequence of positive integers which is denoted by $row(T)^r$.
A skew semi-standard tableau $T$ is called {\it Littlewood--Richardson tableau} (or $\LR$-tableau) if the sequence $row(T)^r$ satisfies the following condition: for each place $j$ the number of $i'$s which occur among the first $j$ elements is greater than or equal to the number of $(i+1)'$s for each positive integer $i$.
As an example, consider the following tableaux of shape $\lambda/\gamma$, where $\gamma=(2,1)$ and $\lambda=(4,3,2)$.
\begin{center}
$T_1
=
\begin{ytableau}
  \none & \none & 2 & 1 \\
  \none & 1 & 3  \\
  2 & 2
\end{ytableau}\quad
T_2
=
\begin{ytableau}
  \none & \none & 1 & 1 \\
  \none & 1 & 2  \\
  3 & 3
\end{ytableau}
\quad
T_3
=
\begin{ytableau}
  \none & \none & 1 & 1 \\
  \none & 1 & 2  \\
  1 & 3
\end{ytableau}
$
\end{center}
We see that $row(T_1)^r=123122$, $row(T_2)^r=112133$,
and $row(T_3)^r=112131$, so in this example the only $\LR$-tableau is $T_3$.

\begin{theorem}\cite[Theorem~2.8.13]{JK}{\em(Littlewood--Richardson rule)}
Let $m<n$ be natural numbers and $\lambda\vdash n$.
Then
$$\chi^\lambda\downarrow^{S_n}_{S_m\times S_{n-m}}=\sum\limits_{\nu\vdash m, \mu\vdash n-m}c_{\mu,\nu}^\lambda\cdot\chi^\mu\cdot\chi^\nu,$$
where $c^\lambda_{\mu,\nu}$ equals the number of the Littlewood--Richardson tableaux of shape $\lambda/\mu$ and content $\nu$.
\end{theorem}
The coefficients $c^\lambda_{\mu,\nu}$ are called {\it Littlewood--Richardson coefficients}.
It is well-known that $c^{\lambda}_{\mu,\nu}=c^{\lambda}_{\nu,\mu}$ and $c^{\lambda}_{\mu,\nu}=0$ unless $\mu,\nu\subseteq\lambda$. Moreover, if $\mu\subseteq\lambda$, then there exists at least one $\nu$ such that $c^{\lambda}_{\mu,\nu}\neq0$.

The special case where $n-m=1$ is known as a separate useful statement.
\begin{theorem}\cite[Theorem~2.3.4]{JK}{\em(Branching Law)}
Suppose that $\chi^\lambda$ is an irreducible character of $S_n$, where $\lambda\vdash n$ and $n\geqslant2$.
Then for the restriction of $\chi^\lambda$ to the stabilizer $S_{n-1}$ of the point $n$, we have
$$\chi^\lambda\downarrow^{S_n}_{S_{n-1}}=\sum\limits_{\mu\vdash n-1, \mu\subseteq\lambda}\chi^{\mu}.$$
\end{theorem}

The following statement follows from the Littlewood--Richardson rule  and is the most useful tool in further proofs.
\begin{lemma}\label{l:Spec}
Suppose that $\chi^\lambda$ is an irreducible character of $S_n$, where $\lambda\vdash n$.
Suppose that $1\leqslant m<n$ and $\sigma\in S_m\times S_{n-m}\leqslant S_n$. 
Write $\sigma=\pi_1\pi_2$, where $\pi_1\in S_m$ and $\pi_2\in S_{n-m}$.
Then $$\Sp_\lambda(\sigma)=\bigcup\limits_{\nu\vdash m, \gamma\vdash n-m, c^\lambda_{\nu,\gamma}\neq0}\Sp_\nu(\pi_1)\times \Sp_\gamma(\pi_2),$$
where $\times$ on the right means the usual multiplication of numbers.
\end{lemma}
\begin{proof}
By the Littlewood--Richardson rule, we have 
$$\chi^\lambda\downarrow^{S_n}_{S_m\times S_{n-m}}=\sum\limits_{\nu\vdash m, \mu\vdash n-m}c_{\mu,\nu}^\lambda\cdot \chi^\mu\cdot \chi^\nu.
$$
Now the character $\chi^\mu\cdot \chi^\nu$ corresponds to the tensor product of representations. 
If $A$ and $B$ are square diagonalizable matrices over $\mathbb{C}$, then the set of eigenvalues of $A\otimes B$ 
comprises of the pairwise products of eigenvalues of $A$ and $B$.
This proves the assertion.
\end{proof}

Suppose that $\lambda$ and $\mu$ are partitions such that $\mu\subseteq\lambda$.
Denote by $\mathcal{LR}([\lambda/\mu])$ 
the set of all partitions $\nu\vdash|\lambda|-|\mu|$
such that $c^\lambda_{\mu,\nu}\neq0$.

Suppose that $T$ and $S$ are skew diagrams.
Following~\cite{GS18}, we write 
$T\cong S$ if 
$T$ is a translation of $S$ on the  Cartesian plane,
and  we write $T\cong S^\circ$ if $T$ is obtained from $S$ by the composition of  a rotation by $180^\circ$ degrees and a translation.

\begin{lemma}\cite[Lemma~4.4]{BK99},\cite[Lemma~2.3]{GS18}\label{l:unique}
Suppose that $\lambda$ and $\mu$
are partitions such that $\mu\subseteq\lambda$.
We have $|\mathcal{LR}([\lambda/\mu])|=1$ if and only if there exists a partition $\alpha$ of size $|\lambda|-|\mu|$ such that $[\lambda/\mu]\cong[\alpha]$ or $[\lambda/\mu]\cong[\alpha]^\circ$.
\end{lemma}

\begin{lemma}\cite[Lemma~2.3]{GS18}\label{l:nonstandard}
Suppose that $T$ is a skew diagram.
If $m=|T|\geqslant 4$ and $(m-1,1)\in\mathcal{LR}(T)$,
then one of the following holds:
\begin{enumerate}[$(i)$]
\item $T\cong[(m-1,1)]$ or $T\cong[(m-1,1)]^\circ$;
\item $\mathcal{LR}(T)\cap\{(m), (m-2,2), (m-2,1,1)\}\neq\varnothing$.
\end{enumerate}
\end{lemma}

\subsection{Auxiliary results}

In this subsection, we prove auxiliary statements that will be used in the inductive proof of the main result.

\begin{lemma}\label{l:remove-d1}
Suppose that $\lambda$ and $\mu $ are partitions of $n\geqslant13$ such that $\lambda_1\geqslant\lambda'_1$ and $\mu$ contains $k\geqslant2$ distinct elements $1<d_1<\ldots<d_k$.
Denote by $\beta$ the partition of $m:=n-d_1$ which is obtained from $\mu$ by deleting one element equal to $d_1$.
Denote by $\alpha$ the partition of $n-d_1$ which
corresponds to the Young diagram obtained from 
$[\lambda]$ after removing $d_1-1$ times 
the first corner and then once either the first or second corner.
The following statements hold.
\begin{enumerate}[$(i)$]
\item If $\beta=(m)$ and $\alpha\in\{(m), (1^m), (m-1,1), (2,1^{m-2})\}$, then $k=2$ and $\lambda\in\{(n), (n-1,1), (n-2,2), (n-2,1,1), (m, 1^{m-1}), (m+1, 1^{m-2}), (m, 1^{m-2}),(m,2,1^{m-3}),(m-1, 1^{m-2}),(m-1,2,1^{m-3}),(m-1,3,1^{m-3}),(m-1,2,2,1^{m-4})\}$.
\item If $m$ is odd, $\beta=(m-2,2)$, and $\alpha\in\{(m-2,2),(2,2,1^{m-4})\}$, then $d_1=2$ and
$\lambda\in\{(n-2,2), (n-3,3), (n-3,2,1)\}$.
\end{enumerate}
\end{lemma}
\begin{proof}
Note that if $d_1\geqslant 6$,
then $m\geqslant d_2\geqslant 7$. On the other hand, if $d_1\leqslant 5$, then $m=n-d_1\geqslant 8$.
This implies that $m\geqslant 7$ in all cases.

Suppose that $\beta=(m)$. Then $k=2$ and $d_2=m$.
If $\alpha=(m)$, then before the last box was removed the diagram was equal to $[(m+1)]$ or $[(m,1)]$ (see Fig.~\ref{f:remove}). It follows that all other removals occurred in the first row.  Therefore, $\lambda\in\{(n), (n-1,1)\}$. If $\alpha=(m-1,1)$, then 
before the last removal the diagram was equal to
$[(m+1)]$, $[(m-1,2)]$, or $[(m-1,1,1)]$.
Since $m\geqslant 7$, all previous removals occurred in the first row. Therefore, we find that $\lambda\in\{(n-1,1), (n-2,2), (n-2,1,1)\}$.  If $\alpha=(1^m)$, then  
$\lambda_1'\geqslant m=d_2$ and hence $\lambda_1\geqslant d_2$. On the other hand, we know that $\lambda_1\leqslant d_1+1$. This is possible only if $m=d_2=d_1+1$, so $\lambda=(m, 1^{m-1})$. 

Finally, consider the case $\alpha=(2,1^{m-2})$.
Before the last removal, only three Young diagrams are possible: $[(3,1^{m-2})]$, $[(2,1^{m-1})]$, and $[(2,2,1^{m-3})]$ (see Fig.~\ref{f:remove}).
In the first case, we have $\lambda=(d_1+2,1^{m-2})$.
Since $\lambda_1\geqslant\lambda_1'=m-1$, we find that 
$m-3\leqslant d_1\leqslant m-1$. Therefore, we get three diagrams. In the second case,
$m\leqslant\lambda_1'\leqslant\lambda_1\leqslant d_1+1$ and 
hence $d_1=m-1$. Therefore, 
all previous removals occurred in the first row,
so $\lambda_1=m$ and $\lambda=(m,1^{m-1})$. 
Suppose that after $d_1-1$ removals we get the diagram $[(2,2,1^{m-3})]$. Then $d_1+1\geqslant \lambda_1\geqslant m-1$. Since $m-1\geqslant d_1$,
we find that either $d_1=m-1$ or $d_1=m-2$. If $d_1=m-2$, then $\lambda_1=m-1=2+(d_1-1)$, so
all previous removals occurred in the first row and
$\lambda=(m-1,2,1^{m-3})$.
If $d_1=m-1$, then either $\lambda_1=m=2+(d_1-1)$ and hence 
$\lambda=(m,2,1^{m-3})$, or $\lambda_1=m-1$ and
$\lambda\in\{(m-1,3,1^{m-3}),(m-1,2,2,1^{m-4})\}$.

\begin{figure}$
\ytableausetup{boxsize=1.3em}
\begin{ytableau}
\empty &  & \none[\ldots] &   & \bullet \\
\bullet
\end{ytableau}
\quad\begin{ytableau}
\empty &  & \none[\ldots] &   & \bullet \\
 & \bullet \\
 \bullet
\end{ytableau}
\quad
\begin{ytableau}
\empty & \bullet \\
 \\
 \none[\ldots] \\
 \\
 \bullet
\end{ytableau}
\quad
\begin{ytableau}
\empty & & \bullet \\
 & \bullet \\
 \none[\ldots] \\
 \\
 \bullet
\end{ytableau}
$\caption{Cases of diagrams before the last step}\label{f:remove}
\end{figure}

Suppose that $m$ is odd and $\beta=(m-2,2)$. Then $k=2$
and $d_1=2$. Therefore, $n=m+2$. If $\alpha=(2,2,1^{m-4})$,
then $\lambda_1\geqslant\lambda_1'\geqslant m-2$,
so $n\geqslant m+m-4$. This implies that $m\leqslant 6$; a contradiction. It remains to consider the case $\alpha=(m-2,2)$. Then before the last removal
we have only three possible Young diagrams: $[(m-1,2)]$, $[(m-2,3)]$, and $(m-2,2,1)]$. By the definition of $\alpha$, 
we find that $\lambda\in [(n-2,2)]$, $[(n-3,3)]$, and $(n-3,2,1)]$, respectively.
\end{proof}

\begin{lemma}\label{l:column}
Suppose that $n\geqslant8$ and $\sigma\in S_n$ is a permutation with cycle type $\mu\vdash n$.
Suppose that $\mu$ contains $k\geqslant 2$ distinct elements $1<d_1<d_2<\ldots<d_k$.
If $\lambda\vdash n$ and $\lambda_1\geqslant\lambda_1'$, then one of the following statements holds.
\begin{enumerate}[$(i)$]
\item $k=2$ and there exists an integer $m$ such that $\lambda=(m,1^{m-1})$, $d_1=m-1$, and $d_2=m$;
\item $n-\lambda'_1-d_1\geqslant k-1$ and only one element of $\mu$ equals $d_1$;
\item $n-\lambda'_1-d_1\geqslant k$.
\end{enumerate}
\end{lemma}
\begin{proof}
Suppose that only one element of $\mu$ is equal to $d_1$. We show that either $k=2$, $\lambda=(m,1^{m-1})$, $d_1=m-1$, and $d_2=m$ or $n-\lambda'_1-d_1\geqslant k-1$.
Assume to the contrary that $n-\lambda_1'-d_1<k-1$. Then 
$n-\lambda_1-d_1\leqslant n-\lambda_1'-d_1\leqslant k-2$ and hence $2d_1\geqslant 2n-\lambda_1-\lambda_1'-2k+4\geqslant n-1-2k+4$. It follows that $d_1\geqslant\frac{n}{2}-k+\frac{3}{2}$.
If $k\geqslant 3$, then $d_k\geqslant d_1+k-1\geqslant\frac{n}{2}-1+\frac{3}{2}$,
$d_{k-1}\geqslant\frac{n}{2}-2+\frac{3}{2}$, and 
$d_{k-2}\geqslant\frac{n}{2}-3+\frac{3}{2}$.
Since $n\geqslant\sum\limits_{i=1}^k d_i$,
we find that $n\geqslant\frac{3n}{2}-6+\frac{9}{2}$ and hence 
$n\leqslant 3$; a contradiction.  Therefore, we have $k=2$. 
Then $d_1\geqslant\frac{n}{2}-2+\frac{3}{2}=\frac{n-1}{2}$
and $d_2\geqslant\frac{n+1}{2}$.
This implies that $n=2m-1$, $d_1=m-1$, and $d_2=m$.
Now $0=k-2\geqslant n-\lambda_1'-d_1=m-\lambda_1'$,
so $\lambda_1'\geqslant m$. Therefore, $\lambda_1\geqslant \lambda_1'\geqslant m$. This is possible only when $\lambda_1=\lambda_1'=m$ and $\lambda=(m,1^{m-1})$;
we arrive at a contradiction with our assumption.

Suppose that at least two elements of 
$\mu$ equal $d_1$. Assume that $n-d_1-\lambda_1'<k$. 
Then $n-d_1-\lambda_1\leqslant n-d_1-\lambda_1'\leqslant k-1$ and hence $d_1\geqslant\frac{n}{2}-k+\frac{1}{2}$. 
This implies that $d_k\geqslant\frac{n}{2}-1+\frac{1}{2}$ and
$d_{k-1}\geqslant\frac{n}{2}-2+\frac{1}{2}$.
Since $n\geqslant d_1+d_k+d_{k-1}$,
we find that $n\geqslant\frac{3n}{2}-k-3+\frac{3}{2}$.
Therefore, $n\leqslant 2k+3$. On the other hand, we know that
$d_1\geqslant 2$ and $d_i\geqslant 3$ for $i\geqslant2$, so $n\geqslant 2\cdot2+3(k-1)$. Then $4+3k-3\leqslant 2k+3$, so $k\leqslant 2$. This implies that $n\leqslant 2\cdot 2+3=7$; a contradiction.
Therefore, we infer that $n-d_1-\lambda_1'\geqslant k$, as claimed.
\end{proof}

\begin{lemma}\label{l:row}
Suppose that $n\geqslant12$
and there exist $k\geqslant 2$ distinct integers
$1<d_1<d_2<\ldots<d_k$ whose sum does not exceed $n$. Suppose that $\lambda\vdash n$ such that $\lambda_1\geqslant\lambda_1'\geqslant 2$. Denote by $\mu$ the partition of $m:=n-d_1$ which 
corresponds to the Young diagram obtained from $[\lambda]$
by removing the first corner $d_1$ times.
Then the following statements hold.
\begin{enumerate}[$(i)$]
\item If $n-\lambda_1\geqslant k$, then $m-\mu_1\geqslant k$;
\item If $n-\lambda_1<k$, then $\lambda_1-\lambda_2\geqslant d_1$ and $m-\mu_1=n-\lambda_1$.
\end{enumerate}
\end{lemma}
\begin{proof}
If $\lambda_1-\lambda_2\geqslant d_1$, then all $d_1$ boxes were removed in the first row, so $\mu_1=\lambda_1-d_1$ and hence $n-\lambda_1=m-\mu_1$. 

Now we suppose that $\lambda_1-\lambda_2\leqslant d_1-1$.
First, we show that $n-\lambda_1\geqslant k$.
Assume to the contrary that $n-\lambda_1\leqslant k-1$.
Then $\lambda_2\leqslant n-\lambda_1\leqslant k-1$.
Since $d_1\geqslant \lambda_1-\lambda_2+1$,
we find that $d_1\geqslant (n-k+1)+(1-k)+1=n-2k+3$.
Then $d_2\geqslant n-2k+4,\ldots, d_k\geqslant n-2k+k+2$.
So $$n\geqslant d_1+\ldots+d_k\geqslant nk-2k^2+\frac{(k+2)(k+3)}{2}-3.$$ 
This implies that $4k^2+6-(k+2)(k+3)\geqslant n(2k-2)$
and hence $3k^2-5k\geqslant n(2k-2)$. 
Since $n\geqslant d_1+\ldots+d_k$,
we find that $n\geqslant 2+3(k-1)=3k-1$.
Therefore,  we see that $3k^2-5k\geqslant (3k-1)(2k-2)$ and hence $3k-2\geqslant 3k^2\geqslant 6k$; a contradiction.

Therefore, it is true that $n-\lambda_1\geqslant k$.
We shall show that $m-\mu_1\geqslant k$.
First, assume that $k=2$ and $$(\lambda,d_1)\in\{((t,t),t-1), ((t,t,1),t), ((t+1,t),t)\},$$ where $t\in\mathbb{N}$. 
Since $n\geqslant 12$, we find that $t\geqslant 6$
and $d_1\geqslant 5$.
Now $m-\mu_1=n-d_1-\mu_1\geqslant 2t-t-\mu_1=t-\mu_1$.
Therefore, $m-\mu_1\geqslant 2$ if $\mu_1\leqslant t-2$.
It is easy to see that if we remove the first corner five times, then the first row contains at most $t-2$ boxes. So $\mu_1\leqslant t-2$, as required.

Now we assume that if $k=2$, then $$(\lambda,d_1)\not\in\{((t,t),t-1), ((t,t,1),t), ((t+1,t),t)\}.$$ 
Since $\lambda_1-\lambda_2$ boxes were removed in the first row, we infer that $\mu_1\leqslant\lambda_2$.
Therefore, we get that $m-\mu_1=n-d_1-\mu_1\geqslant n-d_1-\lambda_2$.
So it suffices to show that $n-d_1-\lambda_2\geqslant k$.
Assume to the contrary that $n-d_1-\lambda_2\leqslant k-1$.
Then $\lambda_2\geqslant n-d_1-k+1$ and hence 
$\lambda_1\geqslant n-d_1-k+1$. This implies that 
$n\geqslant\lambda_1+\lambda_2\geqslant 2n-2d_1-2k+2$.
So $d_1\geqslant\frac{n}{2}-k+1$. If $k\geqslant3$,
then $n\geqslant d_k+d_{k-1}+d_{k-2}\geqslant\frac{n}{2}+(\frac{n}{2}-1)+(\frac{n}{2}-2)=\frac{3n}{2}-3$ and hence $n\leqslant 6$; a contradiction. Therefore, we get that $k=2$. Then $d_1\geqslant\frac{n}{2}-1$ and $d_2\geqslant\frac{n}{2}$. If $n=2t+1$, then 
$d_1=t$ and $d_2=t+1$. On the other hand, $\lambda_2\geqslant n-d_1-k+1=2t+1-t-2+1=t$. This is possible only if $\lambda=(t,t,1)$ or  
$\lambda=(t+1,t)$.  If $n=2t$, then $d_1\geqslant t-1$, so $d_1=t-1$. Then 
$\lambda_2\geqslant n-d_1-k+1=2t-(t-1)-2+1=t$ and hence $\lambda=(t,t)$.
This contradicts our assumption on $(\lambda,d_1)$ for $k=2$. The lemma is proved.
\end{proof}

\subsection{Particular cases}
In this subsection, we prove results that are devoted to special cases of the main theorem. Recall that our description of the minimal polynomials for a given permutation depends on the number of different lengths of nontrivial cycles in the cyclic decomposition. The case when all non-trivial cycles have the same length has already been studied.
\begin{theorem}\cite[Theorem~1.1]{YS21}\label{th:1}
Let $n$, $r$, and $m$ be natural numbers such that $n\geqslant3$, $r\geqslant2$, and $rm\leqslant n$. Suppose that $\sigma\in S_n$ is a product of $m$ independent cycles of length $r$ and $\rho:S_n\rightarrow GL(V)$ is a nontrivial irreducible representation of $S_n$ over $\mathbb{C}$ corresponding to $\lambda\vdash n$. Denote the minimal polynomial of $\rho(\sigma)$ by $p_\lambda^\mu(x)$.
Then $p_\lambda^\mu(x)\neq x^r-1$ if and only if one of the following statements holds.
\begin{enumerate}[$(i)$]
\item We have $r=n$, $m=1$, and $\lambda=(n-1,1)$; in this case $p_\lambda^\mu(x)=~\frac{x^n-1}{x-1}$.
\item We have $r=n$, $m=1$, and $\lambda=(2,1^{n-2})$; in this case $p_\lambda^\mu(x)=\frac{x^n-1}{x+(-1)^n}$.
\item We have $r=n=6$, $m=1$, and $\lambda=(3,3)$; in this
case $p_\lambda^\mu(x)=\frac{x^6-1}{x^2+x+1}$.
\item We have $r=n=6$, $m=1$, and $\lambda=(2,2,2)$; in this
case $p_\lambda^\mu(x)=\frac{x^6-1}{x^2-x+1}$.
\item We have $n=4$, $\lambda=(2,2)$; and $(r,m,p_\lambda^\mu(x))\in\{(4,1,x^2-1),(3,1,x^2+x+1),(2,2,x-1)\}$.
\item $\lambda=(1^n)$ and $p_\mu^\lambda(x)=x-\operatorname{sgn}(\sigma)$.

\end{enumerate}
\end{theorem}

The following statement may already be known to experts and can be proved in different ways (e.g., using \cite[Theorem~3.3.]{ST89}). 
We present this statement because we believe it sheds light on the nature of the main result, and we also provide a proof for it that is consistent with the spirit of the present paper.
\begin{lemma}\label{l:standard}
Suppose that $n\geqslant 2$ and $\rho:S_n\rightarrow GL(V)$ is the standard representation of $S_n$, i.e. it corresponds to the partition $(n-1,1)$.
Take a permutation $\sigma$ which is the product of independent cycles of lengths $n_1\ldots,n_k$ and denote by $V_\eta(\sigma)=\operatorname{ker}(\rho(\sigma)-\eta\cdot\operatorname{Id_V})$ the eigenspace for 
$\eta$. Then 
\begin{enumerate}[$(i)$]
\item $\dim V_1(\sigma)=k-1$; in particular it is zero iff $\sigma$ is an $n$-cycle;
\item if $k>1$, then 
$\Sp_{(n-1,1)}(\sigma)=\bigcup\limits_{i=1}^k\{\eta\in\mathbb{C}~|~\eta^{n_i}=1\}.$
\item if $k=1$, then 
$\Sp_{(n-1,1)}(\sigma)=\{\eta\in\mathbb{C}\setminus\{1\}~|~\eta^{n_i}=1\}.$
\item if $\eta\in\Sp_{(n-1,1)}(\sigma)\setminus\{1\}$,
then $\dim V_\eta(\sigma)$ equals the number of $n_i$ such that $\eta^{n_i}=1$.
\end{enumerate}
\end{lemma}
\begin{proof}
The assertion is clear if $\sigma=1$. So we can assume that $\sigma\neq1$.

First, we consider the case when $\sigma$ is an $n$-cycle.
If $n\geqslant 3$, then $p_{(n-1,1)}^{(n)}(x)=(x^n-1)/(x-1)$  by Theorem~\ref{th:1}. If $n=2$, then $\rho$ coincides with the sign representation and  $p_{(1,1)}^{(2)}(x)=x+1=(x^2-1)/(x-1)$. So the assertion is true in this case
(this can also be verified by finding $\rho(\sigma)$ and its characteristic polynomial explicitly).

Suppose that $\sigma$ is a product of $k\geqslant 2$ cycles.
Then we can assume that $\sigma\in S_m\times S_{n-m}\leqslant S_n$, where $n>m\geqslant 2$. Write $\sigma=\pi\tau$, where $\pi\in S_m$ is a cycle of length $m$ and $\tau$ is the product of other cycles. Clearly, 
there exist only two partitions $\nu\vdash m$ such that
$\nu\subseteq (n-1,1)$: $(m-1,1)$ and $(m)$.
In the first case, we have $c^\lambda_{\nu,\gamma}\neq0$ only when $\gamma=(n-m)$, while in the second 
$c^\lambda_{\nu,\gamma}\neq0$ for $\gamma=(n-m)$ and $\gamma=(n-m-1,1)$ if $n-m\neq1$.
\begin{center}
$
\ytableausetup{boxsize=1.2em}
\begin{ytableau}
\bullet & \none[\ldots] & \bullet & 1 & \none[\ldots] & 1 \\
 \bullet 
\end{ytableau}\quad
\begin{ytableau}
\bullet & \none[\ldots] & \bullet & 1 & \none[\ldots] & 1 \\
1 
\end{ytableau}
\quad
\begin{ytableau}
\bullet & \none[\ldots] & \bullet & 1 & \none[\ldots] & 1 \\
2 
\end{ytableau}
$
\end{center}
By the Littlewood--Richardson rule, we find that 
if $n-m>1$, then  
$$\chi^\lambda\downarrow^{S_n}_{S_m\times S_{n-m}}=\chi^{(m-1,1)}\chi^{(n-m)}+\chi^{(m)}\chi^{(n-m)}+\chi^{(m)}\chi^{(n-m-1,1)},$$ 
and if $n-m=1$, then  
$\chi^\lambda\downarrow^{S_n}_{S_m\times S_{n-m}}=\chi^{(m-1,1)}\chi^{(n-m)}+\chi^{(m)}\chi^{(n-m)}$.
Therefore, we can classify the eigenvalues of $\rho(\sigma)$
in the following way: the eigenvalues of 
$\pi$ in the standard representation of $S_m$, the eigenvalue 1 with multiplicity one, which corresponds to the summand $\chi^{(m)}\chi^{(n-m)}$, and, if $n-m>1$, eigenvalues of $\tau$ in the standard representation of $S_{n-m}$.
Now the result follows by induction.
\end{proof}

The following statement is devoted to some special cases of Theorem ~\ref{t:main}.

\begin{lemma}\label{l:except}
Suppose that $\sigma\in S_n$ has cycle type $\mu$ and $\lambda\vdash n$. Then the following statements hold.
\begin{enumerate}[$(i)$]
\item If $\lambda=(n-2,2)$, $\mu=(n-2,2)$, and
$n\geqslant7$ is odd, then
$p_\lambda^\mu(x)=(x^{o(\sigma)}-1)/(x+1)$.
\item If $\lambda=(n-2,2)$, $\mu=(m,n-m)$,
where $n\geqslant 7$, $m>n-m\geqslant2$ and either $n-m>2$ or $n$ is even, 
then $p_\lambda^\mu(x)=x^{o(\sigma)}-1$.
\item If $\lambda=(n-2,1,1)$, 
$\mu=(m,n-m)$, where $n\geqslant 8$ and $m>n-m\geqslant2$,
then $p_\lambda^\mu(x)=x^{o(\sigma)}-1$.
\item If $\lambda\in\{(m, 1^{m-1})$, $(m+1, 1^{m-2})$, 
$(m, 1^{m-2})$, $(m,2,1^{m-3})$, $(m-1,1^{m-2})$, $(m-1,2,1^{m-3})$,$(m-1,3,1^{m-3})$, $(m-1,2,2,1^{m-4})\}$, 
$\mu=(m,n-m)$, where $n\geqslant 13$ and $m>n-m\geqslant2$,
then $p_\lambda^\mu(x)=x^{o(\sigma)}-1$.
\item If $\lambda\in\{(n-2, 2),(n-3,3), (n-3, 2, 1)\}$
and $\mu=(n-4,2,2)$, where $n$ is odd and $n\geqslant7$,
then $p_\lambda^\mu(x)=x^{o(\sigma)}-1$.
\end{enumerate}
\end{lemma}
\begin{proof}
Suppose that $\lambda=(n-2,2)$ and 
$\mu=(n-2,2)$. There exist only two partitions of 2: $\alpha=(1^2)$ and $\alpha=(2)$. It is easy to see that 
$c^\lambda_{(1^2),\beta}\neq0$ if and only if $\beta=(n-3,1)$, while
$c^\lambda_{(2),\beta}\neq0$ if and only if $\beta=(n-2)$,
$\beta=(n-3,1)$, $\beta=(n-4,2)$.
\begin{center}
\ytableausetup{boxsize=1.2em}
$\begin{ytableau}
\bullet & 1 & 1 & \none[\ldots] & 1 \\
\bullet & 2
\end{ytableau}\quad
\begin{ytableau}
\bullet & \bullet & 1 & \none[\ldots] & 1 \\
1 & 1
\end{ytableau}
\quad\begin{ytableau}
\bullet & \bullet & 1 & \none[\ldots] & 1 \\
1 & 2
\end{ytableau}
\quad
\begin{ytableau}
\bullet & \bullet & 1 & \none[\ldots] & 1 \\
2 & 2
\end{ytableau}
$
\end{center}
By Lemma~\ref{l:Spec},
we find that 
\begin{multline*}
\Sp_\lambda(\mu)=
\Sp_{(n-2)}((n-2))\cup\Sp_{(n-3,1)}((n-2))\cup
\Sp_{(n-4,2)}((n-2))\\\cup\{-1\}\times\Sp_{(n-3,1)}((n-2)).
\end{multline*}
Now $\Sp_{(n-2)}((n-2))=\{1\}$ and 
$\Sp_{(n-3,1)}((n-2))=\{\eta\in\mathbb{C}\setminus\{1\}~|~\eta^{(n-2)}=1\}$ by Lemma~\ref{l:standard}.
Clearly, $\Sp_{(n-4,2)}((n-2))\subseteq 
\{\eta\in\mathbb{C}~|~\eta^{(n-2)}=1\}$.
If $n$ is even, then we see that $\Sp_\lambda((n-2,2))=\{\eta\in\mathbb{C}~|~\eta^{(n-2)}=1\}$, so $p_\lambda^\mu(x)=x^{n-2}-1=x^{o(\sigma)}-1$.
If $n$ is odd, then we get that $
\Sp_\lambda((n-2,2))=\{\eta\in\mathbb{C}\setminus\{-1\}~|~\eta^{2(n-2)}=1\}$ and hence 
$p_\lambda^\mu(x)=(x^{2(n-2)}-1)/(x+1)$, as claimed.

Suppose that $\lambda=(n-2,2)$ and $\mu=(m,n-m)$, where
$m>n-m\geqslant3$. Then we see that $c^{\lambda}_{(n-m-1,1),(m)}>0$ and $c^{\lambda}_{(n-m-1,1),(m-1,1)}>0$. 
\begin{center}
\ytableausetup{boxsize=1.2em}
$
\begin{ytableau}
\bullet & \bullet & 1 & \none[\ldots] & 1 \\
\bullet & 1
\end{ytableau}
\quad\begin{ytableau}
\bullet & \bullet & 1 & \none[\ldots] & 1 \\
\bullet & 2
\end{ytableau}
\quad
\begin{ytableau}
\bullet & \bullet & \bullet & \none[\ldots] & 1 \\
1 & 1
\end{ytableau}
\quad
\begin{ytableau}
\bullet & \bullet & \bullet & \none[\ldots] & 1 \\
1 & 2
\end{ytableau}
$
\end{center}
By Lemmas~\ref{l:Spec} and~\ref{l:standard},
we get that 
$$\{\eta\in\mathbb{C}\setminus\{1\}~|~\eta^{n-m}=1\}\times
\{\zeta\in\mathbb{C}~|~\zeta^m=1\}\subseteq\Sp_\lambda(\mu).$$
Since $c^{\lambda}_{(n-m),(m)},c^{\lambda}_{(n-m),(m-1,1)}>0$,
we infer that $\{\zeta\in\mathbb{C}~|~\zeta^m=1\}\subseteq\Sp_\lambda(\mu)$.
Therefore, 
$\{\eta\in\mathbb{C}~|~\eta^{n-m}=1\}\times
\{\zeta\in\mathbb{C}~|~\zeta^m=1\}\subseteq\Sp_\lambda(\mu)$.
This implies that $\Sp_\lambda(\mu)=\{\eta\in\mathbb{C}~|~\eta^{o(\sigma)}=1\}$ and $p_\lambda^\mu(x)=x^{o(\sigma)}-1$, as required.

Suppose that $\lambda=(n-2,1,1)$ and $\mu=(m,n-m)$, where $n\geqslant 8$ and $m>n-m\geqslant2$. Note that $m\geqslant 5$.
Since  $c^\lambda_{(n-m-1,1),(m)}>0$
and $c^\lambda_{(n-m-1,1),(m-1,1)}>0$ (see~Fig.~\ref{f:n-2-1-1}), 
Lemmas~\ref{l:Spec} and~\ref{l:standard} imply that
$$\{\eta\in\mathbb{C}\setminus\{1\}~|~\eta^{n-m}=1\}\times
\{\zeta\in\mathbb{C}~|~\zeta^m=1\}\subseteq\Sp_\lambda(\mu).$$
Now since $c^\lambda_{(n-m), (m-2,1,1)}>0$ and 
$\Sp_{(m-2,1,1)}((m))=\{\eta\in\mathbb{C}~|~\eta^m=1\}$,
we infer that 
$\Sp_\lambda(\mu)=\{\eta\in\mathbb{C}~|~\eta^{o(\sigma)}=1\}$ and $p_\lambda^\mu(x)=x^{o(\sigma)}-1$.

\begin{figure}$
\ytableausetup{boxsize=1.2em}
\begin{ytableau}
\bullet &  \none[\ldots] & \bullet & 1 & \none[\ldots] & 1 \\
\bullet \\
1
\end{ytableau}
\quad\begin{ytableau}
\bullet & \none[\ldots] & \bullet & 1 & \none[\ldots] & 1 \\
\bullet \\
2
\end{ytableau}
\quad
\begin{ytableau}
\bullet & \none[\ldots] & \bullet & 1 & \none[\ldots] & 1 \\
2 \\
3
\end{ytableau}
$\caption{Tableaux for $\lambda=(n-2,1,1)$}\label{f:n-2-1-1}
\end{figure}

Suppose that $\lambda\in\{(m, 1^{m-1})$, $(m+1, 1^{m-2})$, 
$(m, 1^{m-2})$, $(m-1,1^{m-2})$, $(m,2,1^{m-3})$, $(m-1,2,1^{m-3})$, $(m-1,3,1^{m-3})$, $(m-1,2,2,1^{m-4})\}$, where $|\lambda|\geqslant13$
and $\mu=(d_2, d_1)$, where $m=d_2>d_1>1$.
We see that in all cases $m-1\leqslant \lambda_1'\leqslant m$
and $0\leqslant\lambda_1'-d_1\leqslant 2$.
If $\lambda\neq(m-1,1^{m-2})$, then put $\alpha=(d_1-\lambda_1'+4, 1^{\lambda_1'-4})$, and if $\lambda=(m-1,1^{m-2})$, then put
$\alpha=(d_1-\lambda_1'+5, 1^{\lambda_1'-5})$.
Then $\alpha$ is a partition of $d_1$ such that $\alpha\subseteq\lambda$:
we remove three or four bottom boxes in the first column of $[\lambda]$ and then add remaining $d_1-\lambda_1'+3$ or  $d_1-\lambda_1'+4$ boxes in the first row to obtain $\alpha$. Then $c^\lambda_{\alpha,\beta}\neq0$,
where $\beta$ corresponds to the content of an $\LR$-tableau of shape $\lambda/\alpha$ when we write 1 everywhere in the first row and write 1, 2, 3 
if $\lambda\neq(m-1,1^{m-2})$ or 1, 2, 3, 4 if $\lambda=(m-1,1^{m-2})$ in the first column, while the remaining boxes (if any) are filled as desired
(see~Fig.~\ref{f:hooks} with examples $\lambda=(m,1^{m-1})$ and $(m-1,2,2,1^{m-4})$, where $m=7$, and $\lambda=(m-1,1^{m-2})$, where $m=8$). 
Then we find that $\alpha_1,\alpha_1',\beta_1,\beta_1'\geqslant 3$
in all cases. By Lemma~\ref{l:Spec}, we find that 
$\Sp_\alpha((d_1))\times\Sp_\beta((d_2))\subseteq\Sp_\lambda(\mu)$. By Theorem~\ref{th:1},
we know that $\Sp_\alpha((d_1))=\{\eta\in\mathbb{C}~|~\eta^{d_1}=1\}$ and $\Sp_\beta((d_2))=\{\eta\in\mathbb{C}~|~\eta^{d_2}=1\}$.
Therefore, $\Sp_\lambda(\mu)=\{\eta\in\mathbb{C}~|~\eta^{o(\sigma)}=1\}$ and hence $p_\lambda^\mu(x)=x^{o(\sigma)}-1$, as required.
\begin{figure}$
\ytableausetup{boxsize=1em}
\begin{ytableau}
\bullet &  \bullet & \bullet & 1 & 1 & 1 &  1 \\
\bullet \\
\bullet \\
\bullet \\
1 \\
2 \\
3
\end{ytableau}\quad
\begin{ytableau}
\bullet &  \bullet & \bullet  & \bullet & 1 &  1 \\
\bullet & 1 \\
\bullet & 2 \\
1 \\
2 \\
3
\end{ytableau}
\quad
\begin{ytableau}
\bullet &  \bullet & \bullet & 1 & 1 & 1 &  1 \\
\bullet \\
\bullet \\
1 \\
2 \\
3 \\
4
\end{ytableau}
$\caption{Examples of $\LR$-tableaux}\label{f:hooks}
\end{figure}

Suppose that $\lambda\in\{(n-2, 2),(n-3,3), (n-3, 2, 1)\}$
and $\mu=(n-4,2,2)$, where $n\geqslant7$ is odd.
Write $\sigma=\pi\tau$, where $\pi=(1,2)(3,4)$ and $\tau=(5,6,\ldots, n)$.
It is easy to see that $(3,1)\subset\lambda$ and $c^\lambda_{(3,1),(n-4)}, c^\lambda_{(3,1),(n-5,1)}>0$. 
\begin{center}$
\ytableausetup{boxsize=1.2em}
\begin{ytableau}
\bullet & \bullet & \bullet & 1 & 1 \\
\bullet & 2
\end{ytableau}\quad
\begin{ytableau}
\bullet & \bullet & \bullet & 1 \\
\bullet & 1 & 2
\end{ytableau}
\quad\begin{ytableau}
\bullet & \bullet & \bullet & 1 \\
\bullet & 1  \\
2
\end{ytableau}$
\end{center}
By Lemma~\ref{l:Spec}, we infer that
$$\Sp_{(3,1)}((2,2))\times\Sp_{(n-4)}((n-4))
\cup \Sp_{(3,1)}((2,2))\times\Sp_{(n-5,1)}((n-4))\subseteq\Sp_\lambda(\mu).$$
By Lemma~\ref{l:standard}, $\Sp_{(3,1)}((2,2))=\{1,-1\}$ and 
$\Sp_{(n-5,1)}((n-4))=\{\eta\in\mathbb{C}\setminus\{1\}~|~\eta^{n-4}=1\}$.
This implies that $\{\eta\in\mathbb{C}~|~\eta^{2(n-4)}=1\}\subseteq\Sp_\lambda(\mu)$
and hence $p_\lambda^\mu(x)=x^{2(n-4)}-1$, as required.
\end{proof}

\section{Proof of Theorem~\ref{t:main}} 
\begin{proof}[Proof of Theorem~\ref{t:main}]
Denote by $H$ the cyclic subgroup of $S_n$ generated by $\sigma$.
If $\eta\in\mathbb{C}$ and $\eta^{o(\sigma)}=1$, then denote by 
$\psi_\eta$ the irreducible character of $H$ such that $\psi_\eta(\sigma^i)=\eta^i$,
where $0\leqslant i\leqslant o(\sigma)-1$. Then it is easy to see that the dimension of eigenspace of $\rho(\sigma)$ associated to $\eta$ equals $\langle\chi^\lambda\downarrow H, \psi_\eta\rangle$. 
Thus, for small $n$ we can easily check whether this dimension is zero or not. By going through all possible $\eta$, we can find the minimal polynomial
$p_\lambda^\mu(x)$. Using this strategy, we verify the assertion of the theorem for all $n\leqslant 13$ in the computer algebra system GAP~\cite{GAP}\footnote{The code can be found at the following link: \url{https://github.com/AlexeyStaroletov/SymmetricGroup/blob/main/Eigenvalues/min\_poly.g}}. 

Since $\chi^{\lambda'}=\chi^\lambda\cdot\chi^{(1^n)}$,
the formula for the dimension of eigenspaces also shows that $\eta\in\Sp_\lambda(\mu)$ if and only if $\sgn(\sigma)\eta\in\Sp_{\lambda'}(\mu)$.

Now we prove the statement by induction on $n$. We assume that $n\geqslant 14$. If $\sigma=1$, then $p_\lambda^\mu(x)=x-1=x^{o(\sigma)}-1$ and hence the theorem is true in this case. So we can assume that $\sigma\neq1$.
Therefore, there exist $k\geqslant 1$ distinct elements $d_1,\ldots,d_k$ in $\mu$ such that $1<d_1<\ldots<d_k$. If $k=1$, then the assertion is true by Theorem~\ref{th:1}. Let $k\geqslant2$. Note that the assertion is symmetric with respect to $\lambda$ and $\lambda'$, so we can assume that $\lambda_1\geqslant\lambda'_1$.
By~Lemma~\ref{l:standard}, we can assume that $\lambda\neq(n-1,1)$. 
By~Lemma~\ref{l:except}, if
$k=2$ and $\mu=(m,n-m)$, where $m>n-m\geqslant2$, then we can assume that
$\lambda\not\in\{(n-2,2), (n-2,1,1)\}$
and $\lambda\not\in\{(m, 1^{m-1})$, $(m+1, 1^{m-2})$, 
$(m, 1^{m-2})$, $(m,2,1^{m-3})$, $(m-1,1^{m-2})$, $(m-1,2,1^{m-3})$,$(m-1,3,1^{m-3})$, $(m-1,2,2,1^{m-4})\}$. Similarly, if 
$\mu=(n-4,2,2)$, where $n$ is odd, then 
the cases  $\lambda\in\{(n-2, 2),(n-3,3), (n-3, 2, 1)\}$
are covered by Lemma~\ref{l:except}.

Consider $\sigma$ as an element of $S_{d_1}\times S_{n-d_1}\leq S_n$.
Write $\sigma=\pi\tau$, where $\pi\in S_{d_1}$ is a cycle of length $d_1$
and $\tau\in S_{n-d_1}$. Let $\delta$ denote the partition of $n$ that corresponds to the cycle type of $\tau$, that is, it is obtained from $\mu$ by removing 
one element equal to $d_1$. Denote by $\nu$ the partition of $m:=n-d_1$ which corresponds to the Young diagram obtained from $[\lambda]$ by removing the first corner $d_1$ times. 
Denote by $\alpha$ (if exists) the partition of $m$ which corresponds to the Young diagram obtained from $[\lambda]$ by removing the first corner $d_1-1$ times and then the second corner once.

Note that neither $(\nu,\delta)$ nor $(\alpha,\delta)$ coincides with the exceptional pairs $(\lambda,\mu)$ from items $(vii)-(xi)$ of Theorem since otherwise $n=m+d_1\leqslant 12$.
By Lemma~\ref{l:remove-d1} and assumptions on $\lambda$, if $\delta=(s)$, then neither $\nu$ nor $\alpha$ belongs to $\{((s),(1^s),(s-1,1),(2,1^{s-2})\}$,
and if $\delta=(s-2,2)$, where $s$ is odd, then neither $\nu$ nor $\alpha$ belongs to $\{(s-2,2), (2,2,1^{s-4})\}$.

It easy to see that to prove the assertion for $\sigma$, it suffices to show that if $t:=n-\lambda_1$, then 
\begin{equation}\label{eq:1}
\Sp_\lambda(\mu)=A(\mu,t):=\{
\eta_{1}\cdots\eta_{t}~|~\eta_j^{\mu_{i_j}}=1, \text{ where } 
1\leqslant j\leqslant t\text{ and }\mu_{i_j}\in\mu\}.   
\end{equation}
In particular, if $t\geqslant k$, then $p_\lambda^\mu(x)=x^{o(\sigma)}-1$. 
We prove (\ref{eq:1}) by considering two cases:
$t\leqslant k-1$ and $t\geqslant k$.

{\bf Case~1}. Suppose that $t\leqslant k-1$.
Since $\lambda\neq(n),(n-1,1)$, we infer that $t\geqslant 2$.
By Lemma~\ref{l:row}(ii), we find that $\lambda_1-\lambda_2\geqslant d_1$. This implies that the diagram $[\nu]$ is obtained from $[\lambda]$ by removing $d_1$ boxes in the first row. 

First, we show that $\Sp_\lambda(\mu)\subseteq A(\mu,t)$.
By Lemma~\ref{l:Spec}, it is sufficient to verify that
if $\gamma\vdash m$, $\omega\vdash n-m=d_1$, and 
$c^\lambda_{\gamma,\omega}\neq0$,
then $\Sp_\gamma(\delta)\times\Sp_\omega((d_1))\subseteq A(\mu,t)$. 
Fix such partitions $\gamma$ and $\omega$. By assumption, $(\lambda,\mu)\neq((s,1^{s-1}), (s,s-1))$, where $s\in\mathbb{N}$, 
therefore $n-d_1-\lambda_1'$ is at least the number of distinct elements of $\delta$ that are greater than 1 by Lemma~\ref{l:column}.
Then the same is true for $m-\gamma_1'$ since $\gamma_1'\leqslant\lambda_1'$. 
Note that $\gamma_1\geqslant\lambda_1-d_1$
and $\gamma_1=\lambda_1-d_1$ only if $\gamma=\nu$. 
Therefore, $m-\gamma_1\leqslant m-\lambda_1+d_1=t$.
If $\gamma=\nu$, then $c^\lambda_{\nu,\beta}\neq0$ only if $\beta=(d_1)$. Applying induction,
we find that $\Sp_{\gamma}(\delta)\subseteq A(\delta,t)$.
This implies that
$$\Sp_{\gamma}(\delta)\times\Sp_{(d_1)}((d_1))\\=\Sp_{\gamma}(\delta)\subseteq A(\delta,t)\subseteq A(\mu,t).$$
If $\gamma\neq\nu$, then $m-\gamma_1\leqslant t-1$ and hence, by induction,
$$\Sp_{\gamma}(\delta)\times\Sp_{\beta}((d_1))
\subseteq A(\delta,t-1)\times\{\eta\in\mathbb{C}~|~\eta^{d_1}=1\}\subseteq A(\mu,t).$$
Therefore, it is true that $\Sp_\lambda(\mu)\subseteq A(\mu,t)$.
Now we show the reverse inclusion.
Since $c^\lambda_{\nu,(d_1)}\neq 0$, Lemma~\ref{l:Spec} and induction imply that
$$\Sp_{\lambda}(\mu)\supseteq\Sp_{\nu}(\delta)\times\Sp_{(d_1)}((d_1))=\Sp_{\nu}(\delta)=A(\delta,t).$$
Since $\lambda_1-\lambda_2\geqslant d_1$, we see that
$[\alpha]$ is obtained from $[\lambda]$ by removing $d_1-1$ times the corner in the first row, and then some corner in another row.
This implies that $c^\lambda_{\alpha,(d_1-1,1)}\neq0$ (see Fig.~\ref{f:longrow}).
Since $m-\alpha_1=n-d_1-(\lambda_1-d_1+1)=t-1$
and $m-\alpha_1'\geqslant m-\lambda_1'$, induction implies that
$$
\Sp_\lambda(\mu)\supseteq\Sp_{\alpha}(\delta)\times\Sp_{(d_1-1,1)}((d_1))=A(\delta,t-1)\times\{\eta\in\mathbb{C}\setminus\{1\}~|~\eta^{d_1}=1\}.$$
It is clear that 
$A(\delta,t)\cup[A(\delta,t-1)\times\{\eta\in\mathbb{C}\setminus\{1\}~|~\eta^{d_1}=1\}]=A(\mu,t)$. Therefore, $A(\mu,t)\subseteq\Sp_\lambda(\mu)$ and hence $\Sp_\lambda(\mu)=A(\mu,t)$, as required.

{\bf Case~2}. Suppose that $t\geqslant k$.
Then $n-\lambda_1'\geqslant n-\lambda_1\geqslant k$. 
So we need to prove that $p_\lambda^\mu(x)=x^{o(\sigma)}-1$.
By Lemma~\ref{l:row}, we have $m-\nu_1\geqslant k$.
By Lemma~\ref{l:column}, 
we find that $m-\nu_1'$ is at least the number of distinct elements of $\delta$ that are greater than 1. 
By induction, we know that $p_\nu^\delta(x)=x^{o(\tau)}-1$ and
$\Sp_\nu(\delta)=\{\eta\in\mathbb{C}~|~\eta^{o(\tau)}=1\}$.
If at least two elements of $\mu$ are equal to $d_1$,
then $d_1=o(\pi)$ divides $o(\tau)$ and $o(\sigma)=o(\tau)$. 
Take any $\gamma\vdash d_1$ such that $c^\lambda_{\nu,\gamma}\neq0$.
By Lemma~\ref{l:Spec}, we find that
\begin{center}
$\begin{aligned}
\Sp_\lambda(\mu)\supseteq
\Sp_\nu(\delta)\times\Sp_\gamma((d_1))=&
\{\eta\in\mathbb{C}~|~\eta^{o(\tau)}=1\}\times 
\Sp_\gamma((d_1))\\=&\{\eta\in\mathbb{C}~|~\eta^{o(\tau)}=1\}.
\end{aligned}$
\end{center}
Therefore, $p_\lambda^\mu(x)=x^{o(\sigma)}-1$, as required. 
So we can assume that only one element of $\mu$ equals $d_1$
and hence $\delta$ has exactly $k-1$ distinct elements
that are greater than~1. 

Suppose now that there exists $\gamma\vdash d_1$ such that
$c^\lambda_{\nu,\gamma}\neq0$ and
$$\gamma\not\in\{ (d_1), (1^{d_1}), (d_1-1,1), (2, 1^{d_1-1}), (2^2), (2^3), (3^2)\}.$$
Then $p_\gamma^{(d_1)}(x)=x^{d_1}-1$ by Theorem~\ref{th:1}.
Therefore, Lemma~\ref{l:Spec} implies that
$\Sp_\lambda(\mu)\supseteq\Sp_\nu(\delta)\times\Sp_\gamma((d_1))=
\{\eta\in\mathbb{C}~|~\eta^{o(\sigma)}=1\}$ and hence 
$p_\lambda^\mu(x)=x^{o(\sigma)}-1$, as required.

It remains to consider the case when 
\begin{equation}\label{eq:2}
c^\lambda_{\nu,\gamma}\neq0\text{ implies that }
\gamma\in\{ (d_1), (1^{d_1}, (d_1-1,1), (2, 1^{d_1-1}), (2^2), (2^3), (3^2)\}.    
\end{equation}

It suffices to show that for each $\gamma$ there exist partitions $\kappa\vdash n-d_1$
and $\omega\vdash d_1$ such that $c^\lambda_{\kappa,\omega}\neq 0$,
$\Sp_\kappa(\delta)=\{\eta\in\mathbb{C}~|~\eta^{o(\tau)}=1\}$
and $\Sp_{\omega}((d_1))\cup \Sp_{\gamma}((d_1))=
\{\eta\in\mathbb{C}~|~\eta^{d_1}=1\}$.
Indeed, in this case Lemma~\ref{l:Spec} implies that
$$\Sp_\lambda(\mu)\supseteq\{\eta\in\mathbb{C}~|~\eta^{o(\tau)}=1\}
\times\{\Sp_{\omega}((d_1))\cup \Sp_{\gamma}((d_1))\}=
\{\eta\in\mathbb{C}~|~\eta^{o(\sigma)}=1\}.$$
As the final step of the proof, we show that the required $\kappa$ and $\omega$ exist considering possible cases for $\gamma$. 
In fact, we use $\nu$ or $\alpha$ for $\kappa$ in almost all cases.
Denote by $T$ an $\LR$-tableau of shape $\lambda/\nu$ with content $\gamma$.

{\bf Subcase~2.1}. Suppose that $\gamma=(d_1)$. Then $\Sp_\gamma((d_1))=\{1\}$.
First, assume that the skew diagram $[\lambda/\nu]$ is connected.
Then all its boxes are in the same row. By the definition of $\nu$,
we find that $[\lambda/\nu]$ is the end part of the first row of $[\lambda]$ and $\lambda_1-\lambda_2\geqslant d_1$.
In this case, we take $\kappa=\alpha$. 
Note that $[\alpha]$ is obtained from $[\lambda]$ by removing $d_1-1$ times the corner in the first row, and then some corner in another row. This implies that $c^\lambda_{\alpha,(d_1-1,1)}>0$ (see Fig.~\ref{f:longrow}).
\begin{figure}
\begin{center}
$
\begin{ytableau}
\bullet & \none[\ldots] & \bullet & 1 & 1 & \ldots & 1 \\
\none[\ldots] & \none[\ldots] & \none[\ldots] \\
\none[\ldots] & \none[\ldots] & \bullet \\
\none[\ldots] & \none[\ldots] 
\end{ytableau}
\quad\longrightarrow\quad
\begin{ytableau}
\bullet & \none[\ldots] & \bullet & \bullet & 1 & \ldots & 1 \\
\none[\ldots] & \none[\ldots] & \none[\ldots] \\
\none[\ldots] & \none[\ldots] & 2  \\
\none[\ldots] & \none[\ldots] 
\end{ytableau}
$
\end{center}\caption{Case $\lambda_1-\lambda_2\geqslant d_1$}\label{f:longrow}
\end{figure}
Take $\omega=(d_1-1,1)$. Then $\Sp_\omega((d_1))=\{\eta\in\mathbb{C}\setminus\{1\}~|~\eta^{d_1}=1\}$ by Lemma~\ref{l:standard}. 
Since $m-\alpha_1=m-\nu_1-1=n-d_1-(\lambda_1-d_1)-1=t-1\geqslant k-1$ and $m-\alpha_1'\geqslant m-\nu_1'\geqslant k-1$,
we infer that $\Sp_\alpha(\delta)=\{\eta\in\mathbb{C}~|~\eta^{o(\tau)}=1\}$ by induction.
Therefore, we find the required pair $(\kappa,\omega)$ in this case.

Suppose now that the skew diagram $[\lambda/\nu]$ has more than one connected component. Then, in $T$,
we replace 1 at the right end of the last row with 2 and 
obtain an $\LR$-tableau of shape $\lambda/\nu$ with content $(d_1-1,1)$. 
Since $\Sp_{(d_1-1,1)}((d_1))=\{\eta\in\mathbb{C}\setminus\{1\}~|~\eta^{d_1}=1\}$, we can use $\kappa=\nu$ and $\omega=(d_1-1,1)$ in this case.

{\bf Subcase 2.2}. Suppose that $\gamma=(1^{d_1})$.
Then $\Sp_\gamma((d_1))=\{\sgn(\pi)\}$.

Suppose that $[\lambda/\nu]$ has exactly one connected component.
Then $[\lambda/\mu]$ is a column of length $d_1$.
In this case, we take $\kappa=\alpha$.
Note that $[\alpha]$ is obtained from $[\lambda]$ by removing $d_1-1$ times the corner in 
the same column, and then a corner in another column.
This implies that $c^\lambda_{\alpha,(2,1^{d_1-2})}\neq0$.
\begin{center}
$
\begin{ytableau}
\none[\ldots] & \none[\ldots] & 1 \\
\none[\ldots] & \none[\ldots] & 2 \\
\none[\ldots] & \none[\ldots] & 3  \\
\none[\ldots] & \none[\ldots] & 4 \\
\none[\ldots] & \none[\ldots] 
\end{ytableau}
\quad\longrightarrow\quad
\begin{ytableau}
\none[\ldots] & \none[\ldots] & \bullet \\
\none[\ldots] & \none[\ldots] & 1 \\
\none[\ldots] & \none[\ldots] & 2 \\
\none[\ldots] & \none[\ldots] & 3 \\
\none[\ldots] &  1 
\end{ytableau}
$
\end{center}
Put $\omega=(2,1^{d_1-2})$. 
By Theorem~\ref{th:1}, we see that
$\Sp_{\gamma}((d_1))\cup\Sp_{\omega}((d_1))=\{\eta\in\mathbb{C}~|~\eta^{d_1}=1\}$.
On the other hand, $m-\alpha_1\geqslant m-\nu_1-1\geqslant k-1$
and $m-\alpha_1'\geqslant m-\nu_1'$. By induction, we get that $\Sp_\alpha(\delta)=\{\eta\in\mathbb{C}~|~\eta^{o(\tau)}=1\}$. 
Thus, the pair $(\kappa,\omega)$ satisfies the required conditions.

Suppose that $[\lambda/\nu]$ has at least two components.
Then all components are connected parts of columns.
Consider the last component when we read the word $row(T)^r$.
Replace numbers of the last component 
with $1,2,\ldots$ reading the word $row(T)^r$. 
\begin{center}
$
\begin{ytableau}
\none[\ldots] & \none[\ldots] & 1 \\
\none[\ldots] & \none[\ldots] & 2 \\
\none[\ldots] & \none[\ldots] & 3  \\
4 & \none[\ldots] \\
5 & \none[\ldots] 
\end{ytableau}
\quad\longrightarrow\quad
\begin{ytableau}
\none[\ldots] & \none[\ldots] & 1 \\
\none[\ldots] & \none[\ldots] & 2 \\
\none[\ldots] & \none[\ldots] & 3  \\
1 & \none[\ldots] \\
2 & \none[\ldots] 
\end{ytableau}
$
\end{center}
Then we obtain
an ${\LR}$-tableau of shape $\lambda/\nu$ with content $\omega$,
where $[\omega]$ has two columns. 
If $\Sp_\omega((d_1))\supseteq\{\eta\in\mathbb{C}\setminus\{\sgn(\pi)\}~|~\eta^{d_1}=1\}$, then the pair $(\nu,\omega)$ is required.
By Theorem~\ref{th:1}, this is not true only for $\omega=(2,2), (2,2,2)$. In these cases, $d_1$ is even, so $\sgn(\pi)=-1$.
By the construction of $\omega$, the last component is filled with numbers 1, 2 or 1, 2, 3, respectively. Then we replace numbers with $1,3$ and $1, 2, 4$, respectively. This corresponds to contents $(2,1,1)$ and $(2,2,1,1)$, which we use for $\omega$. Then $\Sp_\omega((d_1))\supseteq\{\eta\in\mathbb{C}\setminus\{-1\}~|~\eta^{d_1}=1\}$ and hence the pair $(\nu,\omega)$ is required.

{\bf Subcase 2.3}. Suppose that $\gamma=(d_1-1,1)$. Then we can assume that $d_1\geqslant 3$ since the case $d_1=2$ corresponds to Subcase~2.2. 
By Theorem~\ref{th:1}, we know that $\Sp_\gamma((d_1))=\{\eta\in\mathbb{C}\setminus\{1\}~|~\eta^{d_1}=1\}$. 

Suppose that $|\mathcal{LR}([\lambda/\nu])|=1$.
Then there exists a partition $\theta$ such
that $[\lambda/\nu]\cong [\theta]$ or $[\lambda/\nu]\cong[\theta]^\circ$
by Lemma~\ref{l:unique}. It is easy to see that $\theta=\gamma$.
Assume that $[\lambda/\nu]\cong [\theta]$. By the construction of $\nu$, we find that $[\nu]$ is obtained from 
$[\lambda]$ by removing $d_1-1$ boxes in the end of the first row and one box in the end of the second. 
We see that $c^\lambda_{\alpha,(d_1)}\neq0$, where $\omega=(d_1)$.
\begin{center}
$
\begin{ytableau}
\bullet & \none[\ldots] & \bullet & 1 & 1 & \none[\ldots] & 1 \\
\none[\ldots] & \none[\ldots] & \bullet & 2 \\
\none[\ldots] & \none[\ldots] & \none[\ldots] \\
\none[\ldots] & \none[\ldots] & \bullet
\end{ytableau}
\quad\longrightarrow\quad
\begin{ytableau}
\bullet & \none[\ldots] & \bullet & \bullet & 1 & \none[\ldots] & 1 \\
\none[\ldots] & \none[\ldots] & \bullet & 1 \\
\none[\ldots] & \none[\ldots] & \none[\ldots] \\
\none[\ldots] & \none[\ldots] & 1
\end{ytableau}
$
\end{center}
As above, we find that $m-\alpha_1=m-\nu_1-1\geqslant k-1$
and $m-\alpha_1'\geqslant m-\nu_1'\geqslant k-1$, so by induction
$\Sp_\alpha(\delta)=\{\eta\in\mathbb{C}~|~\eta^{o(\tau)}=1\}$.
Since $1\in\Sp_{(d_1)}((d_1))$, the pair $(\alpha,\omega)$ satisfies the required conditions.

Assume that $[\lambda/\nu]\cong [\gamma]^\circ$. By the construction of $\nu$, we find that $d_1=3$ and $[\nu]$ is obtained from 
$[\lambda]$ by removing boxes in the first and second rows.
If $\alpha$ exists, then we see that $c^\lambda_{\alpha,\omega}\neq0$, where $\omega=(1^3)$.
\begin{center}
$
\begin{ytableau}
\bullet & \none[\ldots] & \bullet & 1 \\
\none[\ldots] & \none[\ldots] & 1 & 2 \\
\none[\ldots] & \bullet \\
\end{ytableau}
\quad\longrightarrow\quad
\begin{ytableau}
\bullet & \none[\ldots] & \bullet & 1 \\
\none[\ldots] & \none[\ldots] & \bullet & 2 \\
\none[\ldots] & 3 \\
\end{ytableau}
$
\end{center}
Since $\Sp_{(1^3)}((3))=\{1\}$, we infer that the pair $(\alpha,\omega)$ is required.

Assume that $[\lambda/\nu]\cong [\gamma]^\circ$ and $\alpha$ does not exist. Then $\lambda=(t,t)$, where $t\geqslant 7$.
Take $\kappa=(t, t-3)$ and $\omega=(d_1)$.
\begin{center}
$
\begin{ytableau}
 \none[\ldots] & \bullet & \bullet & 1 \\
\none[\ldots] & \bullet & 1 & 2 \\
\end{ytableau}
\quad\longrightarrow\quad
\begin{ytableau}
\none[\ldots] & \bullet & \bullet & \bullet & \bullet \\
\none[\ldots] & \bullet & 1 &  1 & 1  \\
\end{ytableau}
$
\end{center}

Now $m-\kappa_1=m-\nu_1-1\geqslant k-1$.
Since $t\geqslant 7$, we see that
$\kappa\not\in\{(s),(1^s),(s-1,1), (2,1^{s-2}), (s-2,2), (2,2,2^{s-4})\}$, where $s\in\mathbb{N}$. 
By induction, we get that $\Sp_\kappa(\delta)=\{\eta\in\mathbb{C}~|~\eta^{o(\tau)}=1\}$. Therefore, the pair $(\kappa,\omega)$ satisfies the required conditions.

Suppose that $|\mathcal{LR}([\lambda/\nu])|>1$.
Then there exists $\omega\in\{(2, 1^{d_1-1}), (2^2), (2^3), (3^2)\}$
such that $c^\lambda_{\nu,\omega}\neq0$. 
If $\omega\neq(2, 1^{d_1-1})$ or
$\omega=(2, 1^{d_1-1})$ and $d_1$ is even, 
then $1\in\Sp_\omega((d_1))$ by Theorem~\ref{th:1},
so we can take $\kappa=\nu$ and $\omega$. 
Therefore, we can assume that $\omega=(2,1^{d_1-1})$, where $d_1$ is odd.
Since $\omega\neq\gamma$, we infer that $d_1\geqslant5$.
By Lemmas~\ref{l:unique} and \ref{l:nonstandard}, 
we find that $\mathcal{LR}([\lambda/\nu])\cap\{(d_1),(d_1-2,2),(d_1-2,1, 1)\}\neq\varnothing$. By (\ref{eq:2}), we know that 
$\mathcal{LR}([\lambda/\nu])\cap\{(d_1-2,2),(d_1-2,1,1)\}=\varnothing$. 
This implies that $(d_1)\in\mathcal{LR}([\lambda/\nu])$, however this case was considered above.

{\bf Subcase 2.4}. Suppose that $\gamma=(2, 1^{d_1-2})$. Then $\Sp_\gamma((d_1))=\{\eta\in\mathbb{C}\setminus\{\sgn(\pi)\}~|~\eta^{d_1}=1\}$. 
By previous cases, we get that $d_1\geqslant 4$.
Similarly to Subcase 2.3, we can assume that $|\mathcal{LR}([\lambda/\nu])|=1$. 
By Lemma~\ref{l:unique}, there exists a partition $\theta$ such that
$[\lambda/\nu]\cong[\theta]$ or $[\lambda/\nu]\cong[\theta]^\circ$.
By the construction of $\nu$, we find that $\theta=\gamma$ and $[\lambda/\nu]$
 is located at the end of the first $d_1-1$ rows of $[\lambda]$.
If $[\lambda/\nu]\cong[\theta]$,
then $c^\lambda_{\alpha,\omega}\neq0$, where $\omega=(2,2,1^{d_1-2})$, as illustrated in the example below.

\begin{center}
$
\begin{ytableau}
\bullet & \none[\ldots] & \bullet & 1 & 1 \\
\bullet & \none[\ldots] & \bullet & 2 & \none \\
\bullet & \none[\ldots] & \bullet & 3 & \none \\
\bullet & \none[\ldots] & \bullet  \\
\end{ytableau}
\quad\longrightarrow\quad
\begin{ytableau}
\bullet & \none[\ldots] & \bullet & \bullet & 1 \\
\bullet & \none[\ldots] & \bullet & 1 & \none \\
\bullet & \none[\ldots] & \bullet & 2 & \none \\
\bullet & \none[\ldots] & 2   \\
\end{ytableau}
\quad
\begin{ytableau}
\bullet & \none[\ldots] & \bullet & 1 \\
\bullet & \none[\ldots] & \bullet & 2  \\
\bullet & \none[\ldots] & 1 & 3  \\
\bullet & \none[\ldots]   \\
\end{ytableau}
\quad\longrightarrow\quad
\begin{ytableau}
\bullet & \none[\ldots] & \bullet & 1 \\
\bullet & \none[\ldots] & \bullet & 2  \\
\bullet & \none[\ldots] & \bullet & 3  \\
\bullet & 4    \\
\end{ytableau}
$
\end{center}

By Theorem~\ref{th:1}, we get that $\sgn(\pi)\in\Sp_{\omega}((d_1))$.
As above, we see that $m-\alpha_1=m-\nu_1-1\geqslant k-1$ and $m-\alpha_1'\geqslant m-\nu_1'\geqslant k-1$, so by induction
$\Sp_\alpha(\delta)=\{\eta\in\mathbb{C}~|~\eta^{o(\tau)}=1\}$. This implies that the pair $(\alpha,\omega)$ satisfies the required conditions.

If $[\lambda/\nu]\cong[\theta]^\circ$ and $\alpha$ exists, then $c^\lambda_{\alpha,\omega}\neq0$, where $\omega=(1^{d_1})$, as illustrated in the example.
Then the pair $(\alpha,\omega)$ satisfies the required conditions.

Let $[\lambda/\nu]\cong[\theta]^\circ$ and $\alpha$ do not exist. Then $[\lambda]$ has $d_1-1$ rows.
Take $\kappa=(\lambda_1,\lambda_2-1,\ldots,\lambda_{d_1-3}-1, \lambda_{d_1-2}-2, \lambda_{d_1-1}-2)$. We see that $c^\lambda_{\kappa,\omega}\neq 0$, where $\omega=(2,2,1^{d_1-4})$.

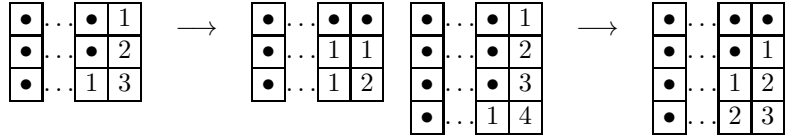
\begin{center}
$
\begin{ytableau}
\bullet & \none[\ldots] & \bullet & 1  \\
\bullet & \none[\ldots] & \bullet & 2 \\
\bullet & \none[\ldots] & 1 & 3  
\end{ytableau}
\quad\longrightarrow\quad
\begin{ytableau}
\bullet & \none[\ldots] & \bullet & \bullet  \\
\bullet & \none[\ldots] & 1 & 1 \\
\bullet & \none[\ldots] & 2 & 2  \\
\end{ytableau}
\quad
\begin{ytableau}
\bullet & \none[\ldots] & \bullet & 1  \\
\bullet & \none[\ldots] & \bullet & 2 \\
\bullet & \none[\ldots] & \bullet & 3  \\
\bullet & \none[\ldots] & 1 & 4   \\
\end{ytableau}
\quad\longrightarrow\quad
\begin{ytableau}
\bullet & \none[\ldots] & \bullet & \bullet  \\
\bullet & \none[\ldots] & \bullet & 1 \\
\bullet & \none[\ldots] & 1 & 2  \\
\bullet & \none[\ldots] & 2 & 3   \\
\end{ytableau}
$
\end{center}
By Theorem~\ref{th:1}, we find that $\sgn(\pi)\in\Sp_\omega((d_1))$.
Since $[\kappa]$ has at least three rows and at least three columns, we infer that $\kappa\not\in\{(s),(1^s),(s-1,1), (2,1^{s-2}), (s-2,2), (2,2,2^{s-4})\}$, where $s\in\mathbb{N}$. Now $m-\kappa_1=n-\nu_1-1\geqslant k-1$ and $m-\kappa_1'\geqslant m-\nu_1'\geqslant k-1$.
By induction, $\Sp_{\kappa}(\delta)=\{\eta\in\mathbb{C}~|~\eta^{o(\tau)}=1\}$. Therefore, the pair $(\kappa,\omega)$ satisfies the required conditions.

{\bf Subcase 2.5}. Suppose that $\gamma=(2,2)$. Then 
$d_1=4$ and $\Sp_\gamma((4))=\{1,-1\}$. 
By previous cases and $(\ref{eq:2})$,
we find that $|\mathcal{LR}([\lambda/\nu])|=1$.
By Lemma~\ref{l:unique}, we see that
$[\nu]$ is obtained from $[\lambda]$ by removing 
two boxes in the end of first row and 
two boxes in the end of second row. 
Then we take $\kappa=\alpha$ and $\omega=(3,1)$.
It is easy to see that $c^\lambda_{\kappa,\omega}\neq 0$ (see Fig.~\ref{f:2-2}).
By Theorem~\ref{th:1}, we find that $p^{(4)}_{(3,1)}(x)=(x^4-1)/(x-1)$.
As above, we see that $m-\alpha_1=m-\nu_1-1\geqslant k-1$
and $m-\alpha_1'\geqslant m-\nu_1'\geqslant k-1$, so by induction
$\Sp_\alpha(\delta)=\{\eta\in\mathbb{C}~|~\eta^{o(\tau)}=1\}$.
This implies that the pair $(\kappa,\omega)$ satisfies the required conditions.

{\bf Subcase 2.6}. Suppose that $\gamma=(2,2,2)$ or $\gamma=(3,3)$.
Then $d_1=6$.

Suppose that $|\mathcal{LR}([\lambda/\nu])|>1$.
By previous cases and $(\ref{eq:2})$,
we find that $c^\lambda_{\nu,\omega}\neq0$,
where $\omega\neq\gamma$ and $\omega\in\{(2,2,2),(3,3)\}$.
By Theorem~\ref{th:1}, we find that $\Sp_\gamma((6))\cup\Sp_\omega((6))=\{\eta\in\mathbb{C}~|~\eta^6=1\}$.
Therefore, in this case we can use this $\omega$ and $\kappa=\nu$.

Suppose that $|\mathcal{LR}(\lambda/\nu)|=1$.
By Lemma~\ref{l:unique}, we find that 
$[\lambda/\mu]\cong[\gamma]$.
Then we take $\kappa=\alpha$. Put $\omega=(4,2)$ if $\gamma=(3,3)$
and $\omega=(3,2,1)$ if $\gamma=(2,2,2)$.
Then $c^\lambda_{\nu,\omega}\neq0$
and the pair $(\kappa,\omega)$ is as required (see Fig.~\ref{f:2-2}). 
\begin{figure}$
\ytableausetup{boxsize=1.2em}
\begin{ytableau}
\bullet & \none[\ldots] & \bullet & \bullet & 1 \\
\bullet & \none[\ldots] & \bullet & 1 & 2 \\
 \none[\ldots] &  \none[\ldots] & \none[\ldots]  \\
\bullet & \none[\ldots] & 1  \\
\none[\ldots] & \none[\ldots] 
\end{ytableau}\quad
\begin{ytableau}
\bullet & \none[\ldots] & \bullet & \bullet & 1  & 1 \\
\bullet & \none[\ldots] & \bullet & 1 & 2 & 2 \\
 \none[\ldots] &  \none[\ldots] & \none[\ldots]  \\
\bullet & \none[\ldots] & 1 \\
\none[\ldots] & \none[\ldots]
\end{ytableau}
\quad\begin{ytableau}
\bullet & \none[\ldots] & \bullet & \bullet  & 1 \\
\bullet & \none[\ldots] & \bullet & 1 & 2 \\
\bullet & \none[\ldots] & \bullet & 2 & 3 \\
\none[\ldots] &  \none[\ldots] & \none[\ldots]  \\
\bullet & \none[\ldots] & 1 \\
\none[\ldots] & \none[\ldots] 
\end{ytableau}
$\caption{Cases $\gamma\in\{(2,2),(2,2,2),(3,3)\}$}\label{f:2-2}
\end{figure}
\end{proof}

\section{Proof of Theorem~\ref{t:main2}}
\begin{proof}[Proof of Theorem~\ref{t:main2}]
We prove the statement by induction on $n$.
Suppose that either $\lambda\neq\lambda'$
or $\lambda=\lambda'$ and $\mu\neq(a_1,\ldots,a_l)$,
where $a_i$ are the lengths of hooks along the main diagonal of $[\lambda]$. The restrictions on $\sigma$ and $\lambda$ imply that for all $i\geqslant0$ we have $\chi(\sigma^i)=\chi^\lambda(\sigma^i)/d$, where $d\in\{1,2\}$. From the formula for the dimension of the eigenspaces of $\rho(\sigma)$, which was discussed at the beginning of the proof of Theorem~\ref{t:main}, it follows that $\sigma$ has the same set of eigenvalues for $\rho$ and for the representation of $S_n$ affording $\chi^\lambda$. Thus, the assertion follows from Theorem~\ref{t:main}.

Suppose now that $\lambda=\lambda'$ and $\mu=(a_1,\ldots,a_l)$.
If $n\leqslant 13$, then we verify the assertion using GAP~\cite{GAP}.
So we can assume that $n\geqslant 14$.
Note that $\lambda_1=\frac{a_1+1}{2}$ and hence 
$n-\lambda_1=n-\lambda_1'=\frac{a_1-1}{2}+a_2+\ldots+a_l\geqslant 1+1+\ldots+1=l$. Therefore, we need to show that $p(x)=x^{o(\sigma)}-1$ in this case.

If $l=1$, then $p(x)=x^{o(\sigma)}-1$ by \cite[Corollary~1.2]{YS21} (see also \cite[Theorem~1.2]{V25}).
So we can assume that $l\geqslant 2$.

Recall that $\chi^\lambda\downarrow^{S_n}_{A_n}=\chi^+_\lambda+\chi^-_\lambda$, where, if $\pi\in A_n$
is not conjugate to $\sigma$ and $\sigma^{(1,2)}$, then $\chi^+_\lambda(\pi)=\chi^-_\lambda(\pi)=\chi^\lambda(\pi)/2$.
Denote $m=n-a_l$ and $k=a_l$.
Since $n\geqslant 14$, we infer that $m\geqslant 9$.

Consider $\sigma$ as an element of $A_m\times A_k$
and write $\sigma=\pi\tau$, where $\pi\in A_m$, $\tau\in A_k$.
Now we find the decomposition of $\chi^\lambda\downarrow^{S_n}_{A_{m}}$ into a sum of irreducible characters of $A_m$ in two ways. 

Using transitivity, we consider the restriction from $S_n$ to $A_{m}$ through the following chain of subgroups:
$A_{m}\leq S_{m}\leq S_{m+1}\leq\ldots\leq S_{n-1}\leq S_n$.
By the Branching law, we see that 
$\chi^\lambda\downarrow^{S_n}_{A_{m}}$ is a sum of $\chi^{\nu}\downarrow^{S_{m}}_{A_{m}}$ with some nonzero multiplicities, where $\nu\vdash m$ such that $\nu\subseteq\lambda$ (this also follows from the Littlewood--Richardson rule). Since $m-1\geqslant a_1-1>\frac{a_1+1}{2}=\lambda_1=\lambda_1'$, we see that $\nu\not\in\{(m), (1^{(m)}, (m-1,1), (2,1^{m-2})\}$. On the other hand, we find that
\begin{equation}\label{eq:3}
\chi^\lambda\downarrow^{S_n}_{A_m\times A_k}=\chi^+_\lambda\downarrow^{A_n}_{A_m\times A_k}+\chi^-_\lambda\downarrow^{A_n}_{A_m\times A_k}=\sum_{i\in I}\psi^+_i\phi_i^++\sum_{j\in J}\psi^-_j\phi_j^-,
\end{equation}
where $\psi^+_i,\psi^-_j\in\Irr(A_m)$, $\phi^+_i,\phi^-_j\in \Irr(A_k)$. Therefore, 
$\chi^\lambda\downarrow^{S_n}_{A_m}=\sum_i\phi_i^+(1)\cdot \psi^+_i+\sum_j\phi_j^-(1)\cdot\psi^-_j$. 

Comparing two decompositions for $\chi^\lambda\downarrow^{S_n}_{A_m}$, we infer that $\psi^+_i$ and $\psi^-_j$ do not correspond to partitions $(m)$, $(1^{m})$, $(m-1,1)$, and $(2,1^{m-2})$. By induction, 
the minimal polynomial of the image of $\pi$
in representations affording characters $\psi^+_i$ and $\psi^-_j$ equals $x^{o(\pi)}-1$. 

If $\psi\in\Irr(A_k)$, then denote the set of eigenvalues of the image of $\tau$ in a representation affording the character $\psi$ by $\Sp_\psi(\tau)$.
To prove the theorem, 
it suffices to show that 
$$\bigcup\limits_{i\in I}\Sp_{\phi_i^+}(\tau)=\bigcup\limits_{j\in J} \Sp_{\phi_j^-}(\tau)=\{\eta\in\mathbb{C}~|~\eta^{k}=1\}.$$

Note that if $g\in A_k$, then $\chi^+_\lambda(g)=\chi^\lambda(g)/2=\chi^-_\lambda(g)$. Using the equality~(\ref{eq:3}),
we find that $\{\phi_i^+\}_{i\in I}
=\{\phi_j^-\}_{j\in J}$. 
Therefore, we may consider only the set
$\{\phi_i^+\}$. As above, we see that $\phi_i^{+}$ 
are restrictions to $A_k$ of all $\chi^\nu\in\Irr(S_k)$,
where $\nu\vdash k$ and $\nu\subseteq\lambda$.
If $k\geqslant 5$, then we see that $\nu=(\frac{k+1}{2},\frac{k-1}{2})$ satisfies these conditions. Therefore, there exists $i_1\in I$
such that $\phi_{i_1}^+=\chi^{\nu}\downarrow^{S_k}_{A_k}$.
By induction, we know that $\Sp_{\phi_{i_1}^+}(\tau)=\{\eta\in\mathbb{C}~|~\eta^{k}=1\}$, as required.

If $k=1$, then there is nothing to prove.
It remains to consider the case  $k=3$. 
Then $A_k$ is the cyclic group of order 3.
Note that $\Irr(A_k)$ consists of three irreducible characters:
$\chi_{(3)}$, $\chi_{(2,1)}^+$, and $\chi_{(2,1)}^-$.
It is clear that $\chi_{(3)}(\tau)=1$, $\chi_{(2,1)}^+(\tau)=\omega$, and $\chi_{(2,1)}^-(\tau)=\omega^2$, where $\omega\in\mathbb{C}$ such that $1+\omega+\omega^2=0$. 
Then $\chi^\lambda(\tau)=a\chi_{(3)}(\tau)+b\chi_{(2,1)}^+(\tau)+c\chi_{(2,1)}^-(\tau)$, where $a,b,c\geqslant 0$.
We know that $(3)\subseteq\lambda$ and
$(2,1)\subseteq\lambda$, so $a>0$ and 
$b$ or $c$ is nonzero. Since $\chi^\lambda(\tau)\in\mathbb{Z}$,
we infer that both $b$ and $c$ are positive.
Therefore, there exist $i_1,i_2,i_3\in I$
such that $\phi_{i_1}=\chi_{(3)}$,
$\phi_{i_2}=\chi_{(2,1)}^+$, and 
$\phi_{i_3}=\chi_{(2,1)}^-$.
Since 
$$\Sp_{\phi_{i_1}}(\tau)\cup\Sp_{\phi_{i_2}}(\tau)\cup\Sp_{\phi_{i_3}}(\tau)=\{\eta\in\mathbb{C}~|~\eta^3=1\},$$
the assertion is true in this case.
This completes the proof of the theorem.
\end{proof}

\section{Acknowledgments}
We thank Amrutha P.,  A. Prasad, and Velmurugan S. for their recent works which have inspired and renewed interest in this topic. We also thank A. Prasad and Velmurugan S. for their comments on a first draft of this article, particularly in the case of alternating groups.

\Address

\end{document}